\documentclass{amsart}
\usepackage{amsfonts,amscd,amssymb,amsmath,amsthm}
\usepackage{graphicx,mathrsfs,appendix}
\usepackage{centernot,caption,subcaption,color,url}
\usepackage[T1]{fontenc}
\usepackage{ctable}
\begin{document}

\newtheorem{theorem}{Theorem}[section]
\newtheorem{lemma}[theorem]{Lemma}
\newtheorem{corollary}[theorem]{Corollary}
\newtheorem{conjecture}[theorem]{Conjecture}
\newtheorem{cor}[theorem]{Corollary}
\newtheorem{proposition}[theorem]{Proposition}
\theoremstyle{definition}
\newtheorem{definition}[theorem]{Definition}
\newtheorem{example}[theorem]{Example}
\newtheorem{claim}[theorem]{Claim}
\newtheorem{remark}[theorem]{Remark}

\newenvironment{pfofthm}[1]
{\par\vskip2\parsep\noindent{\sc Proof of\ #1. }}{{\hfill
$\Box$}
\par\vskip2\parsep}
\newenvironment{pfoflem}[1]
{\par\vskip2\parsep\noindent{\sc Proof of Lemma\ #1. }}{{\hfill
$\Box$}
\par\vskip2\parsep}


\newcommand{\R}{\mathbb{R}}
\newcommand{\T}{\mathcal{T}}
\newcommand{\C}{\mathcal{C}}
\newcommand{\G}{\mathcal{G}}
\newcommand{\Z}{\mathbb{Z}}
\newcommand{\Q}{\mathbb{Q}}
\newcommand{\E}{\mathbb E}
\newcommand{\N}{\mathbb N}

\newcommand{\barray}{\begin{eqnarray*}}
\newcommand{\earray}{\end{eqnarray*}}

\newcommand{\beq}{\begin{equation}}
\newcommand{\eeq}{\end{equation}}


\renewcommand{\Pr}{\mathbb{P}}
\newcommand{\as}{\text{a.s.}}
\newcommand{\Prob}{\Pr}
\newcommand{\Exp}{\mathbb{E}}
\newcommand{\expect}{\Exp}
\newcommand{\1}{\mathbf{1}}
\newcommand{\prob}{\Pr}
\newcommand{\pr}{\Pr}
\newcommand{\filt}{\mathcal{F}}
\newcommand{\F}{\mathcal{F}}
\newcommand{\Bernoulli}{\operatorname{Bernoulli}}
\newcommand{\Binomial}{\operatorname{Binom}}
\newcommand{\Beta}{\operatorname{Beta}}
\newcommand{\Binom}{\Binomial}
\newcommand{\Poisson}{\operatorname{Poisson}}
\newcommand{\Exponential}{\operatorname{Exp}}


\newcommand{\link}{\mbox{lk}}
\newcommand{\Deg}{\operatorname{deg}}
\newcommand{\vertexsetof}[1]{V\left({#1}\right)}
\renewcommand{\deg}{\Deg}
\newcommand{\oneE}[2]{\mathbf{1}_{#1 \leftrightarrow #2}}
\newcommand{\ebetween}[2]{{#1} \leftrightarrow {#2}}
\newcommand{\noebetween}[2]{{#1} \centernot{\leftrightarrow} {#2}}
\newcommand{\Gap}{\ensuremath{\tilde \lambda_2 \vee |\tilde \lambda_n|}}
\newcommand{\dset}[2]{\ensuremath{ e({#1},{#2})}}
\newcommand{\EL}{{ L}}
\newcommand{\ER}{{Erd\H{o}s--R\'{e}nyi }}
\newcommand{\zuk}{{\.{Z}uk}}


\newcommand{\frm}{\ensuremath{ 2\log\log m}}
\newcommand{\csubzero}{c_{0}}
\newcommand{\csubone}{c_{1}}
\newcommand{\csubtwo}{c_{2}}
\newcommand{\csubthree}{c_{3}}
\newcommand{\csubstar}{c_{*}}
\newcommand{\INE}{I^{\epsilon}}
\newcommand{\rsp}{1-C\exp(-md^{1/4}\log n)}
\newcommand{\lc}{\ensuremath{ \operatorname{light}(x,y)}}
\newcommand{\hc}{\ensuremath{ \operatorname{heavy}(x,y)}}

\title[Preferential attachment with choice]{Preferential attachment combined with the random number of choices}
\author{Yury Malyshkin}
\address{Tver State University}
\email{yury.malyshkin@mail.ru}
\thanks{The presented work was performed under the State Assignment N3.8032.2017/BCh of the Ministry of Science and Education of Russia. 
}
\date{\today}
\maketitle

\begin{abstract}
We study the asymptotic behavior of the maximal degree in the degree distribution in an evolving tree model combining local choice and Mori's preferential attachment. In the considered model, the random graph is constructed in the following way. At each step, a new vertex is introduced. 
Then, we connect it with one of $d$ possible neighbors, which are sampled from the set of the existing vertices with probability proportional to their degrees plus some parameter $\beta>-1$. The number $d$ will be randomly selected and the vertex with the largest degree is chosen. It is known that the maximum of the degree distribution for non-random $d>2$ has linear behavior and, for $d=2$, asymptotically equals to $n/\ln n$ up to a constant factor. We prove that if $\mathbb{E}d<2+\beta$, the maximal degree has sublinear behavior with the power $\mathbb{E}d/(2+\beta)$ (as in the preferential attachment without choice), if $\mathbb{E}d>2+\beta$, it has linear behavior and if $\mathbb{E}d=2+\beta$ the maximal degree is of order $n/\ln n$. The proof combines standard preferential attachment approaches with martingales and stochastic approximation techniques. We also use stochastic approximation results to get a multidimensional central limit theorem for numbers of vertices with fixed degrees.
\end{abstract}

\section{Introduction}
In the present work, we further explore how the addition of choice (see, e.g., \cite{DSKrM,KR13,MP14,MP13}) affects Mori's preferential attachment model (see \cite{Mori03,Mori05}). The preferential attachment graph model is a time-indexed inductively constructed sequence of graphs, formed in the following way. We start with some initial graph and then on each step we add a new vertex and an edge between it and one of the old vertices, chosen with probability proportional to its degree. The main reasoning for this model is that it allows us to build evolving graph that has power law degree distribution and small diameter (see, e.g., \cite{BR03,DvdHH,CGH17} and chapter 8 of \cite{Hof13}). 

There are different variations of preferential attachment model, in which additional parameters are introduced into the model. For example, in \cite{Mori03} Mori consider a model where each new vertex is chosen with probability proportional to its degree plus some parameter $\beta>-1$. A further generalization of it would be to choose a vertex with probability proportional to some function $w(\cdot)$ of its degree. Such model was considered in \cite{Athreya08}, where asymptotic was found for sublinear ($w(x)\sim x^p,$ $1/2<p<1$) and superlinear ($\sum_{n=1}^{\infty}\frac{1}{w(n)} < \infty$) cases. In both cases, the degree distribution does not follow a power law. Another way to generalize this model was considered in \cite{BCDR07}, where authors analyzed models with fitness, in which the vertex is chosen with probability proportional to the product of its degree and its fitness.
 
In~\cite{MP13}, limited choice was introduced into the preferential attachment model. In this modified model, at each step, $d$ existing vertices are chosen (independently from each other) with probabilities proportional to their degrees. At the same step, a new vertex is connected with the one with the smallest degree. In~\cite{MP13}, it was shown that the maximal degree at time $n$ in such a model grows as $\ln\ln n/\ln d$ with high probability. 

We follow the max-choice model that was introduced in \cite{MP14}. In this model a new edge is adjacent to the vertex with the highest degree. In~\cite{MP14}, the exact first-order asymptotic for the maximal degree in this model was obtained and almost sure convergence of the appropriately scaled maximal degree was shown. In particular, for $d>2$ the maximal degree had linear behavior, and for $d=2$ it was of order $n/\ln n$. Furthermore, in \cite{M17} it was shown that for $d>2$ the second (along with third and so on) highest degree has sublinear behavior, which is similar to standard preferential attachment model (see, e.g., \cite{FFF04}). In the current work, we provide asymptotic for the maximal degree in the case of random $d$ for Mori's preferential attachment model. We fix the distribution of $d$ and number $\beta>0$, which are parameters of our model, and show that the maximal degree could be asymptotically linear (when $\E d>2+\beta$), sublinear (when $\E d<2+\beta$) or of order $n/\ln n$ (when $\E d=2+\beta$). In particular, in Mori's preferential attachment model (see \cite{Mori05}), $d=1$ so $\E d$ is always less than $2+\beta$ (since $\beta>-1$) and maximal degree has sublinear behavior (which corresponds to a power law degree distribution). The addition of parameters $d$ and $\beta$ allow us to vary the model (and the limiting empirical degree distribution).

Let describe the max-choice model. Fix $\beta>-1$. Introduce a countable non-random set of vertices $V=\{v_i,\;i\in\mathbb{N}\}$ and a random variable $d$ that takes values in $\mathbb{N}$. Let $d_i$, $i\in\mathbb{N}$, be i.i.d. copies of $d$. Define a sequence of random trees $\{ P_n \}$, $n\in\mathbb{Z}_{+}$, by the following inductive rule. Let $P_0$ be graph which consists of a vertex $v_{1}$ with no edges and $P_1$ be the one-edge tree which consists of vertices $v_1$ and $v_2$ and an edge between them. Given $P_{n}$, we construct $P_{n+1}$ by adding one vertex and drawing one edge in the following way.

First, we add a vertex $v_{n+2}$ to $P_n$. So, we get $V(P_{n+1})=\{v_i,\;i=1,...,n+2\}$. Denote $\mathcal{F}_n=\sigma\{E(P_1),...,E(P_n)\}$, where $E(P_i)$ is the edge set of $P_i$, $i=1,...,n$. Let $X^1_n,\ldots,X^{d_n}_n$ be i.i.d. vertices of $V( P_n )$ chosen with the conditional probability
\[
\Pr \left[
X^1_n = v_i|\mathcal{F}_{n}
\right] = \frac{\deg v_i(n)+\beta}{(2+\beta)n+\beta},\quad v_i\in V(P_n),
\]
where $\deg v_i(n)$ is the degree of $v_i$ in $P_n$ (note that $\sum_{v_i}\deg v_i(n)=2n$).

Second, create a new edge between $v_{n+2}$ and $Y_n,$ where $Y_n$ is whichever of $X^1_n$,...,$X^{d_n}_n$ has the largest degree. In the case of a tie, choose according to an independent fair coin toss (this choice will not affect the degree distribution). So, we get $E(P_{n+1})=E(P_n)\bigcup\{\{v_{n+2},Y_n\}\}.$ 

We call this model the \emph{max-choice Mori's preferential attachment tree model}.
Let $M(n)$ be the degree of the highest degree vertex at time $n$ and define $m_j=\sum_{i=j}^{\infty}\mathbb{P}(d=i){i \choose j}$. The main result, that will be formulated below, holds under a certain condition on the distribution of $d$, specifically, $m_j<C^j$ for some constant $C>0$ and all $j\in\mathbb{N}$. Note that this condition holds for the Poisson distribution and any finite distribution.

Let us formulate our main theorem:
\begin{theorem}
\label{thm:max_degree}
Let $m_j<C^j$ for all $j\in\mathbb{N}$ and some constant $C>0$.

If $\mathbb{E}d<2+\beta$, then for any $\epsilon>0$
\[
\lim_{n\to\infty} \mathbb{P}(n^{\mathbb{E}d/(2+\beta)-\epsilon}<M(n)<n^{\mathbb{E}d/(2+\beta)+\epsilon}) = 1.
\]

If $\mathbb{E}d>2+\beta$, then
\[
\frac{M(n)}{n}\to x^{\ast}\;\text{a.s.}
\]
where $x^{\ast}$ is the unique positive solution of the equation $$\sum_{i=1}^{\infty}\mathbb{P}(d=i)\left(1-\left(1-x/2\right)^i\right)=x$$
in the interval $0\leq x\leq 1$.

If $\mathbb{E}d=2+\beta$, then
\[
\frac{M(n)\ln n}{n}\to\frac{(2+\beta)^2}{m_2}\text{a.s.}
\]
\end{theorem}

Our proof is based on the existence of the \emph{persistent hub}, i.e. a single vertex that in some finite random time becomes the highest degree vertex for all time after (see \cite{DM09} for more details). Using this, instead of analyzing the highest degree over all vertices we effectively only need to analyze the degree of just one vertex. The existence of the persistent hub is stated in the following result. 

\begin{proposition}
\label{prop:persistent_hub}
There exist random variables $N$ and $K$, that are finite almost surely, such that at any time $n \geq N$, $\deg v_{K}(n)=M(n)$ and $M(n)>\deg v_{i}(n)$ for any $i\neq K$.
\end{proposition}

The purpose of this proposition is to simplify analysis of the dynamics of $M(n)$.
Indeed, let $L(n)$ be the number of vertices at time $n$ that has the degree equal to $M(n)$. The effect of Proposition~\ref{prop:persistent_hub} is that for some random and sufficiently large $N< \infty$, $L(n) = 1$ for all $n \geq N.$ The dynamics of $M(n)$ is given by the formula
$$M(n+1)-M(n) = \1\{\deg Y_{n+1}(n)=M(n)\},\;\text{with}$$

\begin{equation}
\label{p_n}
\begin{gathered}
\mathbb{E}(M(n+1)-M(n)|\mathcal{F}_n)\\
=\sum_{i=1}^{\infty}\mathbb{P}(d_{n+1}=i)\left(1 - \left(1-\frac{(M(n)+\beta)L(n)}{(2+\beta)n+\beta}\right)^{i}\right).
\end{gathered}
\end{equation}
From here, we will refer to $\mathbb{E}(M(n+1)-M(n)|\mathcal{F}_n)$ as $p_{n}$. Note that cause $M(n+1)-M(n)$ could only take values $0$ and $1$, $p_{n}$ equals to the probability to increase maximal degree at the $n$-{th} step conditionally on $\mathcal{F}_{n}$.

Before starting the proof, let us describe its main ideas and the structure of the paper. In Section 2, we provide an initial estimate on $M(n)$ using a martingale technique. In Section 3, we use this estimate along with stochastic domination by a linear urn model to prove the existence of a persistent hub and, so, prove Proposition~\ref{prop:persistent_hub}. In Section 4, we use it along with stochastic approximation processes to get almost sure convergence of $M(n)/n$. Then we construct martingales and use convergence theorems for them to prove Theorem~\ref{thm:max_degree}. In Section 5, we apply results on a stochastic approximation and get a central limit theorem for the number of vertices with a fixed degree. The statement of the respective result requires some additional heavy notation and so we state this theorem only in Section 5.   
Finally, in Section 6, we discuss the degree distribution and variations of the model.

\section{Initial estimate}
\label{sec:apriori}
We begin with a lower bound for the growth of the highest degree.
We will frequently use the following lemma of~\cite{Galashin}.
\begin{lemma}
\label{lem:numbers}
Suppose that a sequence of positive numbers $r_n$ satisfies
\[
r_{n+1} = r_n\left(1+\frac{\alpha}{n+x}\right),~n \geq k
\]
for fixed $\alpha > 0,$ $k > 0$ and $x.$  Then $r_n/n^{\alpha}$ has a positive limit.
\end{lemma}
Now, we formulate our initial estimate
\begin{lemma}
\label{lem:starting_low_bound}
There is $\gamma>0$, such that
\(
\inf_{n} M(n)/n^{\gamma} > 0
\)
almost surely.
\end{lemma}
\begin{proof}
Define 
$$C_{n+1}=\frac{n}{n-1/(4(2+\beta))}C_n =\left(1+\frac{1/(4(2+\beta))}{n-1/(4(2+\beta))}\right)C_n,\,n\geq 2,$$
with $C_{2} = 1$. By Lemma~\ref{lem:numbers}, we have that $C_n/n^{1/(2(2+\beta))}$ converges to a positive limit.

We will prove that $C_{n}/M(n)$ is a supermartingale for $n\geq 2$. Suppose it is already proved. By Doob's theorem (see, e.g., Corollary 3, p. 509 of \cite{Shir96}), it converges almost surely to some finite limit. Therefore, there is a random positive variable $A_{n_0}$, such that $\frac{M(n)}{n^{\gamma}}\geq A_{n_0}$ almost surely for $\gamma=\frac{1}{4(2+\beta)}$ and $n\geq 2$. So, $\mathbb{P}(\inf_{n\in\mathbb{N}}\frac{M(n)}{n^{\gamma}}>0)=1$.
Now prove that $C_{n}/M(n)$ is a supermartingale, which concludes our prove.
 
Recall that $p_{n}$ equals to the probability to increase maximal degree at the $n$-{th} step conditionally on $\mathcal{F}_{n}$.
Note that for $n\geq 2$ (and, therefore, $M(n)\geq 2$)
\begin{align*}
p_{n} &=\sum_{i=1}^{\infty}\mathbb{P}(d_{n+1}=i)\left(1 - \left(1-\frac{(M(n)+\beta)L(n)}{(2+\beta)n+\beta}\right)^{i}\right) \\
&\geq\sum_{i=1}^{\infty}\mathbb{P}(d_{n+1}=i)\left(1 - \left(1-\frac{M(n)+\beta}{(2+\beta)n+\beta}\right)^{i}\right) \\
&\geq\sum_{i=1}^{\infty}\mathbb{P}(d_{n+1}=i)\left( \frac{M(n)+\beta}{(2+\beta)n+\beta}\right)  \\
&\geq 
\left( \frac{M(n)-1}{(2+\beta)n-1}\right) 
\geq\left( \frac{M(n)}{2(2+\beta)n}\right).
\end{align*}
If $n\geq 2$, for $1/M(n+1)$ we get
\begin{align*}
\E\left(1/M(n+1) | \mathcal{F}_n \right)&= \E\left(\left.\frac{\1\{M(n+1)=M(n)+1\}}{M(n)+1}
+\frac{\1\{M(n+1)=M(n)\}}{M(n)}\right| \mathcal{F}_n\right)\\
&=\left(\frac{p_{n}}{M(n)+1}+\frac{1-p_{n}}{M(n)}\right)\\
&=\frac{M(n)+1-p_{n}}{M(n)(M(n)+1)}\\
&=\frac{1}{M(n)}\left(1-\frac{p_{n}}{M(n)+1}\right)\\
&\leq\frac{1}{M(n)}\left(1-\frac{p_{n}}{2M(n)}\right)\\
&\leq\frac{1}{M(n)}\left(1-\frac{1}{4(2+\beta)n}\right),
\end{align*}
which concludes the proof.
\end{proof}

\section{Persistent hub}
\label{sec:hub}

Our method of proof is based on the comparison of our model with the standard preferential attachment model, and we use the technique of~\cite{Galashin} developed for the latter model.
We divide the proof of Proposition~\ref{prop:persistent_hub} into two parts. First, we prove that only finitely many vertices have degrees that at some time become maximal. Second, we prove that there are only finitely many times at which two vertices simultaneously have maximal degree. 

Let us introduce some notation:
$$\chi(n)=\min\{i\geq n:\deg v_{n+1}(i)=M(i)\},$$
$$U=\sum_{n=0}^{\infty}\1\{\chi(n)<\infty\},$$
$$\psi_{i,j}(n)=\min_{l\geq n}\{\deg v_{i}(l)=\deg v_{j}(l)\}.$$
Here $U$ is the number of vertices (of $V$) whose degrees were maximal at some moments, $\chi(n)$ is the moment it happens for the vertex $v_n$. 

\begin{lemma}
\label{lem:change_of_leader}
$U$ is finite almost surely.
\end{lemma}

To prove the lemma we first need a result (which is stated below) on a random walk similar to the one from ~\cite{MP14},  that describes the evolution of degrees of two vertices in the preferential attachment model without choices.

Fix $n_0\in\mathbb{N},$ $i\neq j\in\mathbb{N},$ $i,j\leq n_0+1$. Let $n_0=\rho_0(i,j)<\rho_1(i,j)<\rho_2(i,j)<...$ be moments, when either $\deg v_i$ or $\deg v_j$ is changed. Consider the two-dimensional random walk $(S_{n}=S_{n}(i,j)=(\deg v_i(n),\deg v_j(n)),\,n\geq n_0)$.

\begin{lemma}
\label{lem:coupling}
There is a random walk $(T_n=(A_n,B_n),\,n\geq 0)$ on $\mathbb{Z}^2$ started from $(\deg v_i(n_0),\deg v_j(n_0))$ with transition probabilities $\frac{A_n+\beta}{A_n+B_n+2\beta}$ and $\frac{B_n+\beta}{A_n+B_n+2\beta}$ for moving one step right and one step up respectively, such that

$$\min\{\deg v_i(\rho_n),\deg v_j(\rho_n)\}\leq\min\{A_{n},B_{n}\} \text{ for } n\geq 0.$$
\end{lemma}
\begin{proof}
Since walks $S_{\rho_n}$ and $T_n$ moves at the same times ($n\geq 0$), by standard results on a stochastic ordering (see, e.g., \cite{KKB77} Theorem 2), it is enough to show that, for any $n\geq 0$, the probability to increase $\min\{\deg v_i(\rho_n),\deg v_j(\rho_n)\}$ does not exceed the probability to increase $\min\{A_n,B_n\}$.
Without loss of generality, assume that $\deg v_i(n)>\deg v_j(n)$ (if $\deg v_i(n)=\deg v_j(n)$, both probabilities equal $0$).

For any $m\geq n_0$, set
\begin{align*}
  F_{m}&=\sum_{v_{k}\in V(P_{m})}(\deg v_{k}(m)+\beta)\;\1\{\deg v_{k}(m) < \deg v_{i}(m)\} \text{ and } \\
  G_{m}&=\sum_{v_{k}\in V(P_{m})}(\deg v_{k}(m)+\beta)\;\1\{\deg v_{k}(m) \leq \deg v_{j}(m)\},
\end{align*}
and let 
$$p_m^i = \Pr[ \deg v_i(m+1) = \deg v_i(m) +1]$$ 
and 
$$p_m^j = \Pr[\deg v_j(m+1) = \deg v_j(m) + 1].$$

Introduce functions
$$h_m(x)=\sum_{k=1}^{\infty}\mathbb{P}(d=k)\left(\frac{x}{(2+\beta)m+\beta}\right)^k,\;x\in[0,(2+\beta)m+\beta].$$
Note that $h_m(x)$ convex for $x\in[0,(2+\beta)m+\beta]$ and, therefore,
\begin{equation}
\label{h}
h_m(x+y)\geq h_m(x)+yh_m^{\prime}(x)\text{ for }x,x+y\in[0,(2+\beta)m+\beta].
\end{equation}

The probability that $\deg v_i(m)$ increases is at least the probability that $v_{i} \in \{X^1_m$,...,$X^{d_m}_m\}$ and that all the others $X^k_m$ have degrees strictly less than $\deg v_i(m).$ Thus,
\begin{equation}
\label{F}
      p_m^i \geq h_m(F_m+\deg v_i(m)+\beta) -h_m(F_m).
\end{equation}
Likewise, the probability that $\deg v_j(m)$ increases is at most the probability that $v_{j} \in \{X^1_m$,...,$X^{d_m}_m\}$ and $\deg v_j(m) = \max_{1\leq k \leq d_m} \deg X^k_m(m).$  Thus,
\begin{equation}
\label{G}
    p_m^j \leq h_m(G_m) - h_m(G_m - \deg v_j(m)-\beta).
\end{equation}
If $\deg v_i(m) > \deg v_j(m),$ we have that $F_m \geq G_m$ and, since $h_m^{\prime}(x)$ is increasing for $x\in[0,(2+\beta)m]$, $h_{m}^{\prime}(F_m)\geq h_{m}^{\prime}(G_m)$. Hence, by (\ref{h}), (\ref{F}) and (\ref{G})
\begin{align*}
  \frac{p_m^i}{p_m^j}
  &\geq \frac{ h_m(F_m+\deg v_i(m)+\beta) - h_m(F_m)}{h_m(G_m) - h_m(G_m - \deg v_j(m)-\beta)}  \\
  &\geq \frac{ h_m^{\prime} (F_m) (\deg v_i(m)+\beta)}{h_m^{\prime}(G_m)(\deg v_j(m)+\beta)}  \\
  &\geq \frac{ h_m^{\prime} (G_m) (\deg v_i(m)+\beta)}{h_m^{\prime}(G_m)(\deg v_j(m)+\beta)}
  =\frac{ \deg v_i(m)+\beta }{ \deg v_j(m)+\beta }.
\end{align*}
Thus,
\[
  \frac{p_m^j}{p_m^i + p_m^j}
  =\frac{1}{1 + \tfrac{p_m^i}{p_m^j}}
  \leq\frac{1}{1 + \tfrac{ \deg v_i(m)+\beta}{\deg v_j(m)+\beta}}
  =\frac{\deg v_j(m)+\beta}{\deg v_i(m)+ \deg v_j(m)+2\beta},
\]
which concludes the proof of the lemma.
\end{proof}

Introduce the stopping times $\pi(i,j)=\min\{n\geq 0: A_{n}=B_{n}|A_{0}=i, B_{0}=j\}$ and the function $q(i,j)=\mathbb{P}(\pi(i,j)<\infty)$. Although, the arguments of $q$ and $\pi$ are integers, sometimes in estimates we will write nonintegers in arguments meaning the values of the floor function of it (or $1$ if it is less then 1).

We need Corollary 2.1 from~\cite{Galashin}, that gives us the following estimate:
$$q(i,1)\leq\frac{Q(i)}{2^i}$$
for some polynomial function $Q(x)$.

We will now use it to prove Lemma \ref{lem:change_of_leader}.
\begin{proof}
By Lemma~\ref{lem:starting_low_bound}, we get
\(
M(n)\geq M n^{\gamma}
\)
for all $n\geq 0$ and some random $M>0$ almost surely. Hence, at time $n$ there is at least one vertex $u$ with degree not less then $Mn^{\gamma}$ with probability 1. The degree of the vertex $v_{n+1}$ could become maximal only if at some moment it exceeds the degree of this vertex. Therefore,
$$\1\{\chi(n)<\infty\}\leq\1\{\deg v_{n+1}(m)=\deg u(m)\;\text{for some}\;m>n\}.$$
Due to Lemma~\ref{lem:coupling}, for each $n\geq 0$ there is a random walk $(T_m=(A_m,B_m),\,m\geq 0)$ started from $(\deg u(n),1)$ with the transition probabilities $\frac{A_n+\beta}{A_n+B_n+2\beta}$ and $\frac{B_n+\beta}{A_n+B_n+2\beta}$, such that $\min\{\deg u(\rho_m),\deg v_{n+1}(\rho_m)\}\leq\min\{A_m,B_m\}$ for $m\geq 0$, where $n_0=n$. Then 
$$\1\{\deg v_{n+1}(m)=\deg u(m)\;\text{for some}\;m>n\}\leq\1\{A_m=B_m\;\text{for some}\;m>n\}.$$
Fix $C>0$. Obviously, as $\mathbb{P}(\deg u(n)\geq Mn^{\gamma})=1$, there are random walks $(T_{m}^{1}=(A_m^1,B_m^1),\,m\geq 0)$, $(T_{m}^{2}=(A_m^2,B_m^2),\,m\geq 0)$ started from $(Mn^{\gamma},1)$, $(Cn^{\gamma},1)$ respectively with the same transition probabilities, such that
$$\pi_n^1=\min\{m\geq 0:A_m^1=B_m^1\}\leq\min\{m\geq 0:A_m=B_m\}=\pi_n,$$
and on $M>C$
$$\pi_n^2=\min\{m\geq 0:A_m^2=B_m^2\}\leq\pi_n^1.$$
Therefore,
$$\1\{\deg v_{n+1}(m)=\deg u(m)\;\text{for some}\;m>n\}\1\{M>C\}$$
$$\leq\1\{M>C\}\1\{\pi_n<\infty\}\leq\1\{M>C\}\1\{\pi_n^1<\infty\}\leq\1\{M>C\}\1\{\pi_n^2<\infty\}.$$
Then
$$\sum_{n=0}^{\infty}\mathbb{P}(\chi(n)<\infty,M>C)=\sum_{n=0}^{\infty}\mathbb{E}\1\{\chi(n)<\infty\}\1\{M>C\}$$
$$\leq\sum_{n=0}^{\infty}\mathbb{E}\1\{\deg v_{n+1}(m)=\deg u(m)\;\text{for some}\;m>n\}\1\{M>C\} $$
$$\leq\sum_{n=0}^{\infty}\mathbb{E}\1\{\pi_n^2<\infty\}\1\{M>C\}\leq \sum_{n=0}^{\infty}\mathbb{E}\1\{\pi_n^2<\infty\}$$
$$=\sum_{n=0}^{\infty}q(Cn^{\gamma},1)\leq
\sum_{n=0}^{\infty}\frac{Q(Cn^{\gamma})}{2^{Cn^{\gamma}}}<\infty.$$
Hence, by Borel-Cantelli Lemma
$$U\1\{M>C\}=\sum_{n=0}^{\infty}\1\{\chi(n)<\infty\}\1\{M>C\}<\infty\quad a.s.$$
Since $M>0$ almost surely,
$$\mathbb{P}(U<\infty)=\mathbb{P}\left(\bigcap_{C>0}\{U\1\{M>C\}<\infty\}\right)=1.$$
\end{proof}

Now let $J$ denote the set of vertices whose degrees become maximal at some moment. According to Lemma~\ref{lem:change_of_leader}, $J$ is finite almost surely. Introduce stopping times
$$\zeta_0(v_i,v_j)=n_0,$$
$$\zeta_l(v_i,v_j)=\inf\{n>\zeta_{l-1}(v_i,v_j):$$
$$\deg v_i(n-1)\neq\deg v_j(n-1)\,\text{and}\,\deg v_i(n)=\deg v_j(n) \},$$
$$N(v_i,v_j)=\sup\{l:\zeta_{l}(v_i,v_j)<\infty\},$$
$$\xi=\sup\{\zeta_{N(v_i,v_j)}(v_i,v_j)|v_i\in J,v_j\in J\}.$$
Note that almost sure finiteness of $\xi$ implies Proposition~\ref{prop:persistent_hub} cause any vertex, whose degree become maximal at any time is in $J$, and an order of degrees of vertices from $J$ does not change after the moment $\xi$. Thus, to complete the proof of Proposition~\ref{prop:persistent_hub} we need the following lemma:
\begin{lemma}
\label{lem:change_of_leadership}
$\xi$ is finite almost surely.
\end{lemma}
\begin{proof}
Since $J$ is finite almost surely, it is enough to prove that, for any $v_i,v_j\in V$, $N(v_i,v_j)$ is finite almost surely. Due to the coupling from Lemma~\ref{lem:coupling} there is a version of $T$, such that $\min\{\deg v_i(\rho_n),\deg v_j(\rho_n)\}$ is dominated by $\min\{A_{n},B_{n}\}$ for $n\geq 0$. 

Introduce stopping times
$$\zeta^T_0=\max\{i+1,j+1\},$$
$$\zeta^{T}_l=\inf\{n>\zeta^{T}_{l-1}:A_{n-1}\neq B_{n-1}\,\text{and}\,A_n=B_n \},$$
$$R=\sup\{l:\zeta^T_{l}<\infty\}.$$
Then $N(v_i,v_j)\leq R$. 

It is a standard fact about P\'olya urn model that if $T_n=(A_n,B_n)$ starts from a point $(a,b)$, then the fraction $A_{n}/(A_{n}+B_{n})$ tends in law to a random variable $H(a,b)$ as $n$ tends to infinity, where $H(a,b)$ has beta probability distribution:
$$H(a,b)\sim\Beta(a+\beta,b+\beta).$$
(See, e.g., Theorem 3.2 in \cite{M09} or Section 4.2 in \cite{JK77}). Thus, the limit of $A_n/(A_n + B_n)$ exists almost surely, and it takes value $1/2$ with probability 0 for any starting point of the process $T$.  Hence, this fraction can be equal to $1/2$ only finitely many times almost surely, and so $R$ is finite almost surely, which completes the proof.
\end{proof}

\section{Final results}
\label{sec:finald}
In this section, we finish the proof of Theorem~\ref{thm:max_degree}.
Introduce the function
\[
f(x)=\sum_{i=1}^{\infty}\mathbb{P}(d=i)(1-(1-x/(2+\beta))^{i}).
\]
Note that condition $m_j<C^j$ for some $C>0$ implies that all moments of $d$ are finite. Hence, the function $f(x)$ has derivatives of any order in $[0,2+\beta]$. Also we have that 
$$1-(1-x/(2+\beta))^i=-\sum_{j=1}^{i}\left(-\frac{x}{2+\beta}\right)^{j}{i \choose j}.$$
Therefore, by changing the order of summation we get that
\begin{equation}
    \label{eq:f_m_j}
\begin{aligned}
f(x)
&=\sum_{i=1}^{\infty}\mathbb{P}(d=i)\sum_{j=1}^{i}(-1)^{j+1}\left(\frac{x}{2+\beta}\right)^{j}{i \choose j}\\
&=\sum_{j=1}^{\infty}(-1)^{j+1}\left(\frac{x}{2+\beta}\right)^{j}\sum_{i=j}^{\infty}\mathbb{P}(d=i){i \choose j}\\
&=\frac{x}{2+\beta}\sum_{j=1}^{\infty}\left(-\frac{x}{2+\beta}\right)^{j-1}m_j.
\end{aligned}
\end{equation}
From \eqref{p_n} we have
\begin{equation}
\label{eq:p_n_f}
    p_n=f\left(\frac{(M(n)+\beta)L(n)}{n+\frac{\beta}{2+\beta}}\right).
\end{equation}
Recall that $p_n$ is probability to increase $M(n)$ conditional on $\mathcal{F}_n$. If $L(n)=1$ (it holds for all $n>N$ due to Proposition~\ref{prop:persistent_hub}), then $p_n=f\left(\frac{M(n)+\beta}{n+\frac{\beta}{2+\beta}}\right)$. 
Therefore, to analise $p_n$ using function $f(x)$ we need condition $L(n)=1$ for $n$ large enought.

Let $C>0$.
Introduce the events $D(C)=\{L(n)=1 \;\forall n\geq C\}$, and the stopping time
\(
\eta_C=\inf\{n \geq C:L(n) > 1\}.
\)
This is the stopping time since 
$$\{\eta_C> n\}=\{L(k)=1,\; C\leq k\leq n\}\in\F_n.$$
Note that $D(C)=\{N\leq C\}$, where $N$ is a random variable, such that $L_n=1$ for $n>N$. 
By Proposition~\ref{prop:persistent_hub}, $\mathbb{P}(D(C))\to 1$ as $C\to\infty$. Also, $\eta_C$ is the first moment since $C$ such, that $L(n)\neq 1$, and $\{\eta_C=\infty\}=D(C)$. We use $\eta_C$ to build different supermartingales during the proof, in particular we use that
\begin{equation}
	\label{eq:eta_C}
\begin{aligned}
\mathbb{E}\left(\1\{C\leq n<\eta_C\}(M(n+1)-M(n))|\F_n\right)
&\!=\!\1\{C\leq n<\eta_C\}p_n\\
&\!=\!\1\{C\leq n<\eta_C\}f\!\left(\!\frac{M(n)+\beta}{n+\frac{\beta}{2+\beta}}\!\right).
\end{aligned}
\end{equation}

Let us describe the main steps of the proof. The behavior of $M(n)$ is similar to the behavior of numbers $C_n$ such that $C_{n+1}-C_{n}=f(\frac{C_{n}}{n})$. It depends on first term in Taylor expansion of $f(x)$ in powers of $x$. By equation \eqref{eq:f_m_j} this term is equal to $\frac{\E d}{2+\beta}x$. So, we would consider cases when $\frac{\E d}{2+\beta}>1$ and $\frac{\E d}{2+\beta}\leq 1$. In the last one we additionally consider case $\frac{\E d}{2+\beta}=1$.

First, we use stochastic approximation to get almost sure convergence of $M(n)/n$ to the zero set of $f(x)-x$ (Lemma~\ref{lem:zero_set}).
If $\mathbb{E}d>2+\beta$, the zero set consists of points $0$ and $x^{\ast}$, so we use martingale arguments to show that $M(n)/n$ does not converges to $0$ almost surely. If $\mathbb{E}d\leq 2+\beta$, $0$ is the only nonnegative element of zero set, so we get that $M(n)/n\to 0$ almost surely. Then we use it, along with some martingales and Taylor expansion, to prove Theorem~\ref{thm:max_degree} for $\mathbb{E}d<2+\beta$. Finally we prove Theorem~\ref{thm:max_degree} for $\mathbb{E}d=2+\beta$ using additional martingales along with more accurate Taylor expansion.  
 
Set $g(x)=f(x)/x$ for $x\neq 0$. From \eqref{eq:f_m_j} we have
\begin{equation}
    \label{eq:g_m_j}
g(x)=\frac{1}{2+\beta}\sum_{j=1}^{\infty}\left(-\frac{x}{2+\beta}\right)^{j-1}m_j
\end{equation}
for $x\neq 0$. Set 
$$g(0)=\frac{1}{2+\beta}\sum_{j=1}^{\infty}\left(-\frac{0}{2+\beta}\right)^{j-1}m_j=\mathbb{E}d/(2+\beta).$$
If $\mathbb{E}d>2+\beta$, then $g(0)>1$.
From \eqref{eq:eta_C} we get that
\begin{equation}
    \label{eq:g(x)}
\begin{aligned}
\1\{C\leq n<\eta_C\}\frac{p_n}{M(n)}&=\1\{C\leq n<\eta_C\}\frac{1}{M(n)}f\left(\frac{M(n)+\beta}{n+\frac{\beta}{2+\beta}}\right)\\
&=\1\{C\leq n<\eta_C\} \frac{M(n)+\beta}{M(n)\left(n+\frac{\beta}{2+\beta}\right)} g\left(\frac{M(n)+\beta}{n+\frac{\beta}{2+\beta}}\right).
\end{aligned}
\end{equation}
We are interested in behavior of $g(x)$ when $x$ is small. Notice, that $f(0)=0$, $f^{\prime}(0)=\mathbb{E}d/(2+\beta)$, $f(x)<1$ for $x\in[0,2+\beta)$ (in particular, $f(1)<1$) and $f(x)$ is concave for $0\leq x\leq 1< 2+\beta$. Therefore, if $\mathbb{E}d> 2+\beta$, then there is a unique solution $x^{\ast}$ of the equation $f(x)=x$ in the interval $(0,1)$ and $f(x)>x$ for $x\in(0,x^{\ast})$. Since $g(x)=f(x)/x$ and $g(0)=\mathbb{E}d/(2+\beta)$, we have
\begin{equation}
    \label{eq:g_ast}
g(x)>1\text{ for }x\in[0,x^{\ast})\text{ if } \mathbb{E}d>2+\beta.
\end{equation} 
If $\mathbb{E}d\leq 2+\beta$, then
$$g(x)=f(x)/x<\mathbb{E}d/(2+\beta)\leq 1$$
for $x\in(0,1)$. 

We will apply the stochastic approximation framework laid out in~\cite{Pemantle} to show that, for $\mathbb{E}d>2+\beta$, $Z_n := M(n)/n \to x_{\ast}$ almost surely. A process $(Z_n)_{n \geq 0}$ adapted to the filtration $(\filt_n)_{n \geq 0}$ is called a \emph{stochastic approximation process} if it can be decomposed as
\begin{equation}
\label{stoch_app}
	Z_{n+1} - Z_n = \frac{1}{n}\left( F(Z_n) + \xi_{n+1} + R_n \right),
\end{equation}
where $F$ is some function, $\Exp( \xi_{n+1} | \filt_n)=0,$ and $R_n$ is an $(\filt_n)$ adapted process satisfying $\sum_{n \geq 1} n^{-1} |R_n| < \infty$ almost surely. 

From the definition of $M(n)$, we have that
\begin{equation*}
\begin{aligned}
\Exp\left( Z_{n+1} - Z_n \vert \filt_n \right) 
&= \Exp\left( \frac{M(n+1)}{n+1}-\frac{M(n)}{n}+\frac{M(n)}{n+1}-\frac{M(n)}{n+1}\vert \filt_n\right)\\
&= \Exp\left( \frac{-M(n)}{n(n+1)}+\frac{M(n+1) - M(n)}{n+1} \vert \filt_n \right) \\
&=\frac{-Z_n}{n+1} + \frac{f[((M(n)+\beta)L(n))/(n+\beta/(\beta+2))]}{n+1}.
\end{aligned}
\end{equation*}
Set $F(x) = f(x) - x.$ Thus,
$$\Exp\left( Z_{n+1} - Z_n \vert \filt_n \right) = 
	\frac{F(Z_n)}{n+1} + \frac{f[((Z(n)+\beta/n)L(n))/(1+\beta/(n(\beta+2)))]-f(Z_n)}{n+1}.$$   
Recall that function $f(x)$ has bounded derivative in $[0,1]$. Hence there is constant $c>0$ (that depends on $d$ and $\beta$), such that
\begin{equation}
\label{eq:f_diff}    \left|f\left(\frac{Z_n+\beta/n}{1+\beta/((2+\beta)n)}\right)-f(Z_n)\right|\leq\frac{c}{n}.
\end{equation}
Set 
\[
	R_n = n\Exp\left( Z_{n+1} - Z_n \vert \filt_n \right) - F(Z_n),
\]
and note that, for (\ref{stoch_app}) to hold, we then must take $\xi_{n+1} = n(Z_{n+1} - \Exp( Z_{n+1} \vert \filt_n)).$ 

By Proposition 
\ref{prop:persistent_hub}, there is a random variable $N$ so that for all $n \geq N,$ $L_n = 1$ almost surely. To estimate $|R_n|$ we would use \eqref{eq:f_diff} to bound $\Exp\left( Z_{n+1} - Z_n \vert \filt_n \right)$ for $n\geq N$. Also, note that $|F(Z_n)| \leq 1$ almost surely. As result we have that
\begin{align*}
	\sum_{n =2}^\infty \frac{|R_n|}{n}
	&\leq 
	\sum_{n =2}^\infty \1 \{L(n)\neq 1\}+
	\sum_{n =2}^\infty \1 \{L(n) = 1\}\frac{|R_n|}{n}\\
	&=
	\sum_{n =2}^\infty\! \1 \{L(n)\!\neq\! 1\}
	\!+\! 
	\sum_{n = 2}^\infty \!\frac{\left|n\!\left(\!f\!\left(\!\frac{(Z_n+\beta/n)}{1+\beta/((2+\beta)n)}\!\right)\!-\!f(Z_n)\!\right)\! -\!F(Z_n)\right|}{n(n+1)}\1\{L(n)\!=\!1\}\\
	&\leq
	N
	+ \sum_{n = 2}^\infty \left(c+1\right)\left(\frac{1}{n}-\frac{1}{n+1}\right)
	=
	\frac{c+1}{2}+N.
\end{align*}
This is finite almost surely by the finiteness of $N$. Hence, $(Z_n)_{n \geq 0}$ is a stochastic approximation process.

We now use the following corollary of Lemma 2.6 from~\cite{Pemantle}.
\begin{proposition}
	\label{prop:sa}
	If for all $n\in\mathbb{N}$ $\Exp(\xi_{n+1}^2 \vert \filt_n) \leq K$ for some $K > 0$ almost surely and if $F$ is continuous with an isolated zero set, then $Z_n$ converges almost surely to a random variable with values in the zero set.
\end{proposition}

As $0\leq M(n+1)-M(n)\leq 1$, we get
\begin{align*}
\frac{n+1}{n}|\xi_{n+1}|&=(n+1)\left|Z_{n+1}-\mathbb{E}(Z_{n+1}|\mathcal{F}_n)\right|\\
&=
|M(n+1)-\mathbb{E}(M(n+1)|\mathcal{F}_n)|\\
&=|(M(n+1)-M(n))-\mathbb{E}(M(n+1)-M(n)|\mathcal{F}_n)|\leq 1.
\end{align*}
So, the condition $\Exp[\xi_{n+1}^2 \vert \filt_n] \leq K$ from Proposition~\ref{prop:sa} holds.
Hence, we prove 
\begin{lemma}
\label{lem:zero_set}
If $\mathbb{E}d>2+\beta$, $Z_n$ converges almost surely to some random variable $\eta\in \{0,x_{\ast}\}$, and if $\mathbb{E}d\leq 2+\beta$, $Z_n$ converges almost surely to $0$. 
\end{lemma}

\subsection{Case $\mathbb{E}d>2+\beta$}
$ $

At first, consider the case $\mathbb{E}d>2+\beta$. Due to Lemma~\ref{lem:zero_set}, it is enought to prove that $M(n)/n$ does not converge to zero almost surely. 
Therefore, it suffices to prove the following.
\begin{lemma}
\label{lem:interval_est}
For any $n_0>0$ and $0<\epsilon<x^{\ast}$  there is almost surely finite random variable $T$ ($T$ depends on $n_0$ and $\epsilon$), $T\geq n_0$, such that $\Pr(x_{\ast}-\epsilon\leq M(T)/T)\to 1$ as $n_0\to\infty$.
\end{lemma}
Before starting the proof, let describe its main idea. When $\frac{M(n)}{n}<x_{\ast}-\epsilon$, $\frac{p_n}{M(n)}>(1+\alpha)/n$ for some $\alpha>0$. Therefore, for $n$ such that $\frac{M(n)}{n}<x_{\ast}-\epsilon$ we get that $M(n)$ grows like $n^{1+\alpha}$ and, hence, $M(n)/n$ would exceed $x_{\ast}-\epsilon$ at some finite time. We formalize this condition on $n$ by introducing stopping times $\phi_{n_0}$, which would be the first time (since $n_0$) that it breaks.
\begin{proof}

Let $n_0>0, 0<\epsilon<x^{\ast}$. Recall that $p_n$ is the probability that $M(n+1) = M(n) + 1$ conditionally on $\filt_n.$ From \eqref{eq:g(x)} for $n$ such that $n_0 \leq n \leq \eta_{n_0}$ we have
\begin{equation}
\label{eq:g}
\begin{aligned}
\frac{p_{n}}{M(n)}&= \frac{1+\frac{\beta}{M(n)}}{n\left(1+\frac{\beta}{(2+\beta)n}\right)}g\left(\frac{M(n)+\beta}{n+\beta/(2+\beta)}\right)\\
&=\frac{1}{n}g\left(\frac{M(n)+\beta}{n+\beta/(2+\beta)}\right)\left(1+O\left(\frac{1}{M(n)}\right)\right).
\end{aligned}
\end{equation}

Introduce stopping times
$$\phi_{n_0}=\inf\{n\geq n_0:x_{\ast}-\epsilon\leq M_{n}/n\}.$$
Note that $\phi_{n_0}$ is the first moment since $n_0$ that satisfies the statement of the lemma. Therefore, it is sufficient to prove that $\Pr(\phi_{n_0}<\infty)\to 1$ as $n_0\to\infty$.
Consider the expectation
\begin{align*}
\Exp\!\left(\!\left.\1\{n\!<\!\eta_{n_0}\!\wedge\!\phi_{n_0}\}\!\frac{M(n)}{M\!(n\!+\!1)} \!\right|\! \mathcal{F}_{n}\!\right)\!&=\1\{n<\eta_{n_0}\wedge\phi_{n_0}\}\Exp\!\left(\!\left.\frac{M(n)}{M(n+1)} \right| \mathcal{F}_{n}\!\right)\\
&=\!\!
\1\!\{n\!<\!\eta_{n_0}\!\wedge\!\phi_{n_0}\}\!\left(\!p_{n}\frac{M(n)}{M(n)+1}+1-p_{n}\!\right)\\
&=\!\!
\1\!\{n\!<\!\eta_{n_0}\!\wedge\!\phi_{n_0}\}\!\left(\!p_{n}\!\left(\!1-\frac{1}{M(n)+1}\!\right)\!+\!1\!-\!p_{n}\!\right)\\
&=\!\!
\1\!\{n\!<\!\eta_{n_0}\!\wedge\!\phi_{n_0}\}\!\left(1-\frac{p_{n}}{M(n)}(1+O(M(n)^{-1}))\right)\\
&=\!\! \1\!\{n\!<\!\eta_{n_0}\!\wedge\!\phi_{n_0}\}\!\!\left(\!1\!-\!\frac{1}{n}g\!\left(\!\frac{M\!(n)}{n}\!\right)\!(1\!+\!O(\!M(n)^{\!-\!1}\!)\!)\!\right)\!.
\end{align*}
The last equality holds since \eqref{eq:g} holds and $g(x)$ has a bounded derivative on $[0,1]$.
Note that the term $O(M(n)^{-1})$ from last equality is $\mathcal{F}_n$-countable and by Lemma~\ref{lem:starting_low_bound} turns to $0$ as $n\to\infty$. Introduce stopping times  
$$\alpha_{n_0,c}=\inf\left\{n\geq n_0:g\left(\frac{M(n)}{n}\right)\left(1+O\left(M(n)^{-1}\right)\right)<1+c\right\},\;c>0.$$
From \eqref{eq:g_ast} we have that for any $\epsilon>0$ there is $\delta>0$ so that $g(x)>1+2\delta$ if $0\leq x\leq x_{\ast}-\epsilon$. Then $\Pr(\alpha_{n_0,\delta}=\infty)\to 1$ as $n_0\to\infty$. For $n_0\leq n$ we have
$$
\Exp\left(\left.\frac{1}{M(n\!+\!1)}\1\{n\!<\!\eta_{n_0}\!\wedge\!\phi_{n_0}\!\wedge\!\alpha_{n_0,\delta}\} \right| \mathcal{F}_{n} \right)<\frac{(1\!-\!(1\!+\!\delta)/n)}{M(n)}\1\{n\!<\!\eta_{n_0}\!\wedge\!\phi_{n_0}\!\wedge\!\alpha_{n_0,\delta}\}.
$$
Set
$$C_{n+1}=(1+(1+\delta)/n)C_{n},\;n\geq n_{0},\;C_{n_0}=1$$
For $n\geq n_0$ introduce $A(n) = C_n/M(n)$. Then $A(n\wedge\eta_{n_0}\wedge\phi_{n_0}\wedge\alpha_{n_0,\delta})$ is a supermartingale for $n\geq n_0$.
By Lemma~\ref{lem:numbers}, we have that $C_nn^{-1-\delta}$ converges to a positive limit, and,
by Doob's theorem, $A(n\wedge\phi_{n_0}\wedge \eta_{n_0}\wedge\alpha_{n_0,\delta})$ tends to a finite limit with probability 1.
Thus, there is a random constant $B > 0$ so that $M(n) \geq B n^{1+\delta}$ almost surely for all $n_0\leq n \leq \phi_{n_0} \wedge \eta_{n_0}\wedge\alpha_{n_0,\delta}.$  On the other hand, $M(n) \leq 2n,$ and so $\phi_{n_0} \wedge \eta_{n_0}\wedge\alpha_{n_0,\delta} < \infty$ almost surely.  Therefore, on the event $\{\eta_{n_0}\wedge\alpha_{n_0,\delta}= \infty\},$ we have $\phi_{n_0} < \infty$ almost surely. Since $\Pr(\eta_{n_0}\wedge\alpha_{n_0,\delta}=\infty)\to 1$ as $n_0\to\infty$, we have $\Pr(\phi_{n_0} < \infty)\to 1$ as $n\to\infty$. Hence, we could set $T=\phi_{n_0}\1\{\phi_{n_0}<\infty\}+n_0\1\{\phi_{n_0}=\infty\}<\infty$.
\end{proof}

\subsection{Case $\mathbb{E}d\leq 2+\beta$}
$ $

Second, consider the case $\mathbb{E}d\leq 2+\beta$. The idea of the proof is similar to the previous case. Instead of stopping times $\phi_{n_0}$, $\eta_{n_0}$ and $\alpha_{n_0,\delta}$ we would use more complex stopping times $\nu_{n_0,\epsilon}$ to construct supermartingales.  In addition we consider Taylor expansion (formula \eqref{eq:g_m_j}), and then use convergence $M(n)/n\to 0$ almost surely and condition $m_j<C^j$ for some $C>0$ to estimate its tail.

Fix $n_0>0$, $\delta>0$, $n>n_0$. Consider the expectation
\begin{align*}
\mathbb{E}\left(\left.\1\{n<\eta_{n_0}\}\frac{M(n+1)/(n+1)^{\delta}}{M(n)/n^{\delta}}\right|\mathcal{F}_n\right)
&=\frac{M(n)+p_n}{M(n)}\frac{n^{\delta}}{(n+1)^{\delta}}\1\{n<\eta_{n_0}\}\\
&=\left(1+\frac{p_n}{M(n)}\right)\frac{1}{(1+1/n)^{\delta}}\1\{n<\eta_{n_0}\}.
\end{align*}
The factor $\1\{n<\eta_{n_0}\}$ is needed to remove $L(n)$ from formula \eqref{p_n} for $p_n$. We now obtain the expansion of $p(n)/M(n)$ by powers of $M(n)/n$. From \eqref{eq:g_m_j} and \eqref{eq:g(x)}  (for $n>n_0$) we have
\begin{align*}
\1\{n<\eta_{n_0}\}\frac{p_n}{M(n)}&= \1\{n<\eta_{n_0}\} \frac{M(n)+\beta}{M(n)\left(n+\frac{\beta}{2+\beta}\right)} g\left(\frac{M(n)+\beta}{n+\frac{\beta}{2+\beta}}\right)\\
&=
\1\{n<\eta_{n_0}\}\frac{1+\beta/M(n)}{(2+\beta)n+\beta} \sum_{j=1}^{\infty}\!\left(\!-\frac{M(n)\!+\!\beta}{(2\!+\!\beta)n\!+\!\beta}\!\right)^{j-1}m_j\\
&=
\1\{n<\eta_{n_0}\}\frac{1}{(2+\beta)n} \left(\frac{1\!+\!\beta/M(n)}{1\!+\!\beta/((2\!+\!\beta)n)}\right)\left(\mathbb{E}d-\frac{M(n)+\beta}{n+\beta/(2+\beta)}\right.\\
&\times
\left.\left(\frac{m_2}{2+\beta}+\sum_{j=3}^{\infty} \left(-\frac{M(n)+\beta}{n+\beta/(2+\beta)}\right)^{j-2}\frac{m_j}{(2+\beta)^{j-1}}\right)\right).
\end{align*}
By Lemma~\ref{lem:zero_set}, $\frac{M(n)+\beta}{n+\beta/(2+\beta)}\to 0$ almost surely as $n\to\infty$. Thus, by our conditions on $d$, for any fixed constant $c>0$
$$\mathbb{P}\left(\left|\left(-\frac{M(n)+\beta}{n+\beta/(2+\beta)}\right)^{j-2}\frac{m_j}{(2+\beta)^{j-1}}\right|<c^{j-2}\text{ for all }j\geq 3\text{ and }n>N\right)\to 1$$
as $N\to\infty$. Therefore, 
\begin{equation}
\label{eq:p_n_est}
\begin{aligned}
\1\{n<\eta_{n_0}\}\frac{p_n}{M(n)}&=
\1\{n<\eta_{n_0}\}\frac{1}{(2+\beta)n}
\left(1+O\left(\frac{1}{M(n)}\right)\right) \\
&\times \left(\mathbb{E}d-\frac{M(n)}{n}\left(1+O\left(\frac{1}{M(n)}\right)\right)\left(\frac{m_2}{2+\beta}+o(1)\right)\right)\\
&=\1\{n<\eta_{n_0}\}\frac{\mathbb{E}d}{(2+\beta)n}(1+o(1))\text{ as } n\to\infty,
\end{aligned}
\end{equation}
where $o(1)$ denotes some random variable that converges almost surely to $0$ as $n\to\infty$. Hence,
$$
\begin{aligned}
&\mathbb{E}\left(\left.\frac{M(n+1)/(n+1)^{\delta}}{M(n)/n^{\delta}}\1\{n<\eta_{n_0}\}\right|\mathcal{F}_n \right)
\\
&=
\frac{1}{(1+1/n)^{\delta}}\left(1+\frac{\mathbb{E}d}{(2+\beta)n}(1+o(1))\right)\1\{n<\eta_{n_0}\}
\\
&=
\left(1+\frac{\frac{\mathbb{E}d}{(2+\beta)}-\delta}{n}(1+o(1))\right)\1\{n<\eta_{n_0}\}\to 1\;\text{a.s. as }n_0\to\infty.
\end{aligned}
$$
Moreover, for any $\epsilon\in(0,\frac{\mathbb{E}d}{2+\beta}),$
$$\mathbb{P}\left(\mathbb{E}\left( \left.\frac{M(n+1)/(n+1)^{\mathbb{E}d/(2+\beta)-\epsilon}}{M(n)/n^{\mathbb{E}d/(2+\beta)-\epsilon}}\right|\mathcal{F}_n\right)>1,\, \text{for all}\, n>n_0\right)$$
$$\geq\mathbb{P}\left(\bigcap_{n>n_0}\left\{\mathbb{E}\left(\left.\1\{n<\eta_{n_0}\} \frac{M(n+1)/(n+1)^{\mathbb{E}d/(2+\beta)-\epsilon}}{M(n)/n^{\mathbb{E}d/(2+\beta)-\epsilon}}\right|\mathcal{F}_n\right)>1\right\}\right)$$
$$\geq\mathbb{P}\left(\{\eta_{n_0}=\infty\}\bigcap\left(\bigcap_{n>n_0}\left\{1+\frac{\epsilon}{n}(1+o(1))>1\right\}\right)\right)\to 1 \text{ as }n_0\to\infty.$$
Similarly,
$$\mathbb{P}\left(\mathbb{E}\left( \left. \frac{M(n+1)/(n+1)^{\mathbb{E}d/(2+\beta)+\epsilon}} {M(n)/n^{\mathbb{E}d/(2+\beta)+\epsilon}}\right|\mathcal{F}_n\right)<1,\, \text{for all}\, n>n_0\right)\to 1,\text{ as } n_0\to\infty.$$
Let $A(n)=\frac{M(n)}{n^{\mathbb{E}d/(2+\beta)+\epsilon/2}}$ and $B(n)=\frac{n^{\mathbb{E}d/(2+\beta)-\epsilon/2}}{M(n)}$, which are nearly supermartingales. Introduce stopping times
$$\nu_{n_0,\epsilon} =\inf\left\{n>n_0:\mathbb{E}\left(\left. \frac{M(n+1)/(n+1)^{\mathbb{E}d/(2+\beta)-\epsilon}} {M(n)/n^{\mathbb{E}d/(2+\beta)-\epsilon}}\right|\mathcal{F}_n\right)\leq 1, \text{or}\right.$$
$$\left.\mathbb{E}\left(\left. \frac{M(n+1)/(n+1)^{\mathbb{E}d/(2+\beta)+\epsilon}} {M(n)/n^{\mathbb{E}d/(2+\beta)+\epsilon}}\right|\mathcal{F}_n\right)\geq 1 \right\}.$$
Note that $\mathbb{P}(\nu_{n_0,\epsilon}=\infty)\to 1$ as $n_0\to\infty$. Then $A(n\wedge\nu_{n_0,\epsilon})$ and $B(n\wedge\nu_{n_0,\epsilon})$ are supermartingales, and from Doob's theorem $$\frac{M(n)}{n^{\frac{\mathbb{E}d}{2+\beta}-\epsilon}}\to \infty \text{ and } \frac{M(n)}{n^{\mathbb{E}d/(2+\beta)+\epsilon}}\to 0 \text{ a.s., }n\to\infty,$$
which implies our theorem for $\mathbb{E}d<2+\beta$, and give us the following asymptotic: 
$$\frac{M(n)}{n^{1-\delta}}\to \infty\text{ a.s. for any }\delta>0\text{ if }\mathbb{E}d=2+\beta.$$
In particular, if we consider stopping times
$$\tau_{\epsilon} (n_0) = \inf_{n > n_0} \{ n ~:~ M(n) < \epsilon n^{0.76} \},$$
then $\mathbb{P}(\tau_{\epsilon}(n_0)<\infty)\to 0$ as $n_0\to\infty$ (below, we use this statement in the proof of the theorem in the case $\mathbb{E}d=2+\beta$).

\subsection{case $\mathbb{E}d=2+\beta$}
$ $

Finally, consider the case $\mathbb{E}d=2+\beta$. The proof in this case is similar to the proof in the case $d=2$ in \cite{MP14}. In this case, we need to consider additional term in Taylor expansion and build different martingales.

For fixed $c>0$, consider the following set of scaling functions of $M(n)$:
\begin{equation}
\begin{aligned}
Q_n^c &= \exp( cn / M(n) ) / n, \\
U_n^c &= n\exp( -cn / M(n) ).
\end{aligned}
\end{equation}

\begin{lemma}
\label{lem:mgs}
Let $\epsilon > 0$ and $C>0$ be fixed numbers.
\begin{enumerate}
\item
For each $c<(2+\beta)^2/m_2,$ there is a constant $n_1 = n_1(C,c,\epsilon) \geq C$ so that 
\(
Q^c_{n \wedge \tau_\epsilon(C) \wedge \eta_C},~n \geq n_1
\)
is a supermartingale.
\item
For each $c>(2+\beta)^2/m_2,$ there is a constant $n_2 = n_2(C,c,\epsilon) \geq C$ so that
\(
U^c_{n \wedge \tau_\epsilon(C) \wedge \eta_C},~n \geq n_2
\)
is a supermartingale.
\end{enumerate}
\end{lemma}
\begin{proof}[Proof of Lemma~\ref{lem:mgs}]

\noindent \emph{Proof of (1):} 
Recall that $M(n+1)$ equals eihter $M(n)+1$ (with conditional probability $p_n$) or $M(n)$. Hence, by Taylor's theorem for corresponding exponents,
$$
\Exp\left(\left. \1\{n<\eta_C\}\frac{Q^c_{n+1}}{Q^c_n} \right| \mathcal{F}_n \right)$$
$$=\1\{n<\eta_C\}\frac{n}{n+1}\left(
e^{\left( \frac{c}{M(n)}\right)}(1-p_n) + p_ne^{\left(c\tfrac{n+1}{M(n)+1} - \tfrac{cn}{M(n)}\right)}
\right) $$
$$=\1\{n<\eta_C\}\left( 1-\frac{1}{n} +\frac{c}{M(n)} +cp_n\left(\frac{-1}{M(n)} + \frac{M(n) - n}{M(n)(M(n)+1)}\right)\right)$$
$$+\1\{n<\eta_C\}O\left(\frac{1}{M^2(n)} + \frac{n^2p_n}{M^4(n)}\right).$$
As $p_n\leq 1$ and, for $n<\tau_{\epsilon}(C)$, $M(n)/n^{3/4}\to \infty$, the latter error term is $o(1/n).$  By (\ref{eq:p_n_est}), we get
$$\Exp\left(\left. \1\{n<\eta_C\wedge\tau_{\epsilon}(C)\}\frac{Q^c_{n+1}}{Q^c_{n}} \right| \mathcal{F}_n \right)
= \1\{n<\eta_C\wedge\tau_{\epsilon}(C)\}\left(1-\frac{1}{n} +\frac{c}{M(n)} \right.$$
$$-\left.c\left(\frac{n}{M(n)}\right) \left(\frac{1}{(2+\beta)n}\left(\mathbb{E}d-\frac{M(n)}{n}\left(\frac{m_2}{(2+\beta)}+o(1)\right)\right)\right) +O\left(\frac{1}{n^{1.001}}\right)\right)$$ 
$$= \1\{n<\eta_C\wedge\tau_{\epsilon}(C)\}\left(1-\frac{1}{n} +\frac{cm_2+o(1)}{(2+\beta)^2n}+O\left(\frac{1}{n^{1.001}}\right)\right).$$
Hence, when $c<\frac{(2+\beta)^2}{m_2}$, we may find a constant $n_1\geq C$, such that 
$$\mathbb{E}\left(\left. \1\{n_1\leq n<\eta_{C}\wedge\tau_{\epsilon}(C)\}\frac{Q^c_{n+1}}{Q^c_{n}} \right| \mathcal{F}_n \right) \leq 1,$$
which completes the proof.

\noindent \emph{Proof of (2)} As in the proof of (1), it suffices to show that for $c> \frac{(2+\beta)^2}{m_2}$ and some $n_2>C$,
\[
\Exp\left(\left. \1\{n_2\leq n<\eta_{C}\wedge\tau_{\epsilon}(C)\}\frac{U^c_{n+1}}{U^c_n} \right| \mathcal{F}_n \right) \leq 1.
\]
By Taylor's theorem we get
$$\Exp\left(\left. \1\{n<\eta_C\}\frac{U^c_{n+1}}{U^c_n} \right| \mathcal{F}_n \right)$$
$$=\1\{n<\eta_C\}\frac{n+1}{n}\left(
e^{\left( \frac{-c}{M(n)}\right)}(1-p_n) + p_ne^{\left(-c\tfrac{n+1}{M(n)+1} + \tfrac{cn}{M(n)}\right)}
\right).$$
In the same way as in the proof of (1), we get the equality
$$\Exp\left(\left. \1\{n<\eta_C\wedge\tau_{\epsilon}(C)\}\frac{U^c_{n+1}}{U^c_n} \right| \mathcal{F}_n \right) $$
$$=\1\{n<\eta_C\wedge\tau_{\epsilon}(C)\}
\left(1 + \frac{1}{n} - \frac{cm_2+o(1)}{(2+\beta)^2n} + O\left( \frac{1}{n^{1.001}} \right)\right),$$
which concludes the proof.
\end{proof}
Let us prove the main theorem for $\mathbb{E}d=2+\beta$.
\begin{proof}
Set $O^{\tau}_{\epsilon}=\{\tau_{\epsilon}(1) = \infty\}.$ Note that
\(
\lim_{\epsilon \to 0} \Pr(O^{\tau}_{\epsilon}) = 1.
\)

Recall that $D(C)=\{L(n)=1, \;\forall n\geq C\}.$
For $c<\frac{(2+\beta)^2}{m_2},$ by Lemma~\ref{lem:mgs} and Doob's theorem, there is some random $R_\epsilon$, so that
\[
\Pr\left(\left\{\sup_{n > 0} Q^c_n < R_{\epsilon} < \infty\right\}\cap O^{\tau}_{\epsilon}\cap D(C)\right)=\Pr\left(O^{\tau}_{\epsilon}\cap D(C)\right).
\]
Hence,
\[
\Pr\left(\left\{M(n) \geq \frac{cn}{\ln n + \ln R_\epsilon}\right\}\cap O^{\tau}_{\epsilon}\cap D(C)\right)=\Pr\left(O^{\tau}_{\epsilon}\cap D(C)\right),
\]
and so
\[
\Pr\left(\left\{\liminf_{n \to \infty} \frac{M(n) \ln n}{n} \geq c\right\}\cap O^{\tau}_{\epsilon}\cap D(C)\right)=\Pr\left(O^{\tau}_{\epsilon}\cap D(C)\right).
\]
Thus, we have that
\[
\Pr \left(
\bigl\{\liminf_{n \to \infty} \tfrac{M(n) \ln n}{n} \geq c\bigr\}
\cap O^{\tau}_{\epsilon}
\cap D(C)
\right)
= \Pr \left( O^{\tau}_{\epsilon} \cap D_{C}\right),
\]
and so, taking $\epsilon \to 0$ and $C \to \infty$, we have that
\[
\Pr \left(\liminf_{n \to \infty} \frac{M(n) \ln n}{n} \geq c\right)=1
\]
As this holds for any $c<\frac{(2+\beta)^2}{m_2},$ we conclude the desired lower bound.

For the upper bound, we apply the same technique, but for $U_n^c$ (instead of $Q_n^c$). We get
\[
\Pr\left(\left\{\sup_{n > 0} U^c_n < R_{\epsilon} < \infty\right\}\cap O^{\tau}_{\epsilon}\cap D(C)\right)=\Pr\left(O^{\tau}_{\epsilon}\cap D(C)\right).
\]
As above, the last equality implies
\[
\Pr \left(
\left\{\limsup_{n \to \infty} \tfrac{M(n) \ln n}{n} \leq c\right\}
\cap O^{\tau}_{\epsilon}
\cap D(C)
\right)
= \Pr \left( O^{\tau}_{\epsilon} \cap D(C) \right),
\]
and so taking $\epsilon \to 0$ and $C \to \infty$ we have that
\[
\Pr\left(\limsup_{n \to \infty} \frac{M(n) \ln n}{n} \leq c\right)=1.
\]
As this holds for any $c>\frac{(2+\beta)^2}{m_2},$ the proof is complete.
\end{proof}

\section{Number of vertices with fixed degree}
Let $N_k(n)$ denote the number of vertices with the degree $k$ in $P_n$, $k\in\N$. In this section, we prove a central limit theorem for the vector $(N_1(n),...,N_k(n))$ for any fixed $k\in\mathbb{N}$  
under the assumption $\E C^d<\infty$ for any $C>0$. We give the formal statement of the theorem in the final part of the section because it requires some heavy notations which are given below.

The dynamics of $N_k(n)$ depends on the degree of $Y_{n}$. Since the degree of new vertex always equals to one, the case $k=1$ would be different from the case $k>1$. If $\deg Y_n(n)=1$, then $N_2(n+1)=N_2(n)+1$ and $N_i(n+1)=N_i(n)$ for $i\neq 2$. If $\deg Y_n(n)=k>1$, then $N_{1}(n+1)=N_{1}(n)+1$ (we add vertex with degree $1$), $N_{k}(n+1)=N_k(n)-1$ and $N_{k+1}(n+1)=N_{k+1}(n)+1$ (we increase degree of the vertex with degree $k$). Also, $N_{l}(n+1)=N_l(n)$ for $l\neq 1,k,k+1$. Hence, for $k=1$ we get
\begin{equation}
\label{eq:N_1}
\begin{aligned}
&\{N_1(n+1)-N_1(n)=1|\F_n\}=\{\deg Y_n(n)\neq 1|\F_n\},\\
&\{N_1(n+1)-N_1(n)=0|\F_n\}=\{\deg Y_n(n)= 1|\F_n\}.
\end{aligned}
\end{equation}
For $k>1$ we get
\begin{equation}
\label{eq:N_k}
\begin{aligned}
&\{N_k(n+1)-N_k(n)=1|\F_n\}=\{\deg Y_n(n)= k-1|\F_n\},\\
&\{N_k(n+1)-N_k(n)=-1|\F_n\}=\{\deg Y_n(n)= k|\F_n\}.
\end{aligned}
\end{equation}
Hence,
\begin{equation}
\label{eq:exp_N_k}    
\begin{aligned}
&\E(N_{1}(n+1)\!-\!N_1(n)|\F_n)\!=\!1-\Pr(\deg Y_n(n)= 1|\F_n),\\
&\E\!(\!N_{k}(n+1)\!-\!N_k(n)|\F_n\!)\!=\!\Pr\!(\!\deg Y_n(n)\!=\! k\!-\!1\!)\!-\!\Pr\!(\!\deg Y_n(n)\!=\! k|\F_n\!),\;k>1.
\end{aligned}
\end{equation}
Let $\alpha_0(n)=0$ and for $k\geq 1$
$$\alpha_{k}(n)=\sum_{j=1}^{n+1}\1\{\deg v_j(n)\leq k\}\frac{\deg v_j(n)+\beta}{(2+\beta)n+\beta}=\sum_{j=1}^{k}\frac{N_j(n)(j+\beta)}{(2+\beta)n+\beta}.$$
Note that $\mathbb{P}(\deg X_n^1(n)\leq k)=\alpha_k(n)$. Hence
\begin{equation}
\label{eq:Y_deg}
\begin{aligned}
\Pr(\deg Y_n(n)\!=\!k|\F_n)\!&=\!\Pr\!\left(\!\left.\bigcap_{i=1}^{d_n}\left\{\deg X_n^i(n)\!\leq\! k\right\}\!\backslash\! \bigcap_{i=1}^{d_n}\left\{\deg X_n^i(n)\!\leq\! k\!-\!1\right\}\!\right|\!\F_n\!\right)\\
&=\!\sum_{m=1}^{\infty}\Pr(d=m)\left(\left(\alpha_{k}(n)\right)^m-\left(\alpha_{k-1}(n)\right)^m\right).
\end{aligned}
\end{equation}
Introduce functions
$$h_0=1,\;h_1(x_1)=\sum_{m=1}^{\infty}\Pr(d=m)\left(x_1\frac{1+\beta}{2+\beta}\right)^m,$$
$$h_k(x_1,...,x_k)\!=\!\sum_{m=1}^{\infty}\Pr(d=m)\!\left(\!\left(\!\sum_{j=1}^{k}x_{j}\frac{(j+\beta)}{2+\beta}\right)^m\!-\! \left(\sum_{j=1}^{k-1}x_{j}\frac{(j+\beta)}{2+\beta}\right)^m\!\right),\; k\geq 2,$$
$$f_k(x_1,...,x_k)=h_{k-1}(x_1,...,x_{k-1})-h_{k}(x_1,...,x_k),\quad k\in\N.$$
Since $\alpha_k(n)=\sum_{j=1}^{k}\frac{N_j(n)}{n+\frac{\beta}{2+\beta}}\frac{(j+\beta)}{2+\beta}$, from \eqref{eq:Y_deg} we get that for all $k$
\begin{equation}
\label{eq:Y_h_k}
\Pr(\deg Y_n(n)=k|\F_n) =h_k\left(\frac{N_1(n)}{n+\frac{\beta}{2+\beta}},...,\frac{N_1(n)}{n+\frac{\beta}{2+\beta}}\right).
\end{equation}
Therefore, from \eqref{eq:exp_N_k} we have
$$\E(N_k(n)-N_{k-1}(n)|\F_n) = f_k\left(\frac{N_1(n)}{n+\frac{\beta}{2+\beta}},...,\frac{N_1(n)}{n+\frac{\beta}{2+\beta}}\right).$$

Consider the sets 
$\Delta_k=[0,\infty)^k.$
Note that since $\E C^d<\infty$, the functions $f_i(x_1,...,x_k)$, $i=1,...,k$, are well defined and have continuous partial derivatives on $\Delta_k$.
Let $Z_k(n)=\frac{N_k(n)}{n}$ and $\tilde{Z}_k(n)=\frac{N_k(n)}{n+\frac{\beta}{2+\beta}}=Z_{k}(n)\frac{1}{1+\frac{\beta}{n(2+\beta)}}$. Note that $(\tilde{Z}_1(n)...,\tilde{Z}_k(n))\in\Delta_k$ and  $(Z_1(n)...,Z_k(n))\in\Delta_k$. For all $k$, 
$$\E\left(N_{k}(n+1)-N_k(n)|\F_n\right)= f_k\left(\tilde{Z}_1(n),...,\tilde{Z}_k(n)\right).$$
Note that
$$\tilde{Z}_k(n+1)-\tilde{Z}_k(n)=\frac{\left(n+\frac{\beta}{2+\beta}\right)(N_k(n+1)-N_k(n))-N_k(n)} {\left(n+\frac{\beta}{2+\beta}\right)\left(n+1+\frac{\beta}{2+\beta}\right)},$$
and, hence,
$$\E\left(\tilde{Z}_{k}(n+1)-\tilde{Z}_k(n)|\F_n\right) =\frac{f_k(\tilde{Z}_1(n),...,\tilde{Z}_k(n))-\tilde{Z}_k(n)}{n+1+\frac{\beta}{2+\beta}}.$$
Similary,
$$Z_{k}(n+1)-Z_k(n)=\frac{N_{k}(n+1)-N_k(n)}{n+1}-\frac{Z_k(n)}{n+1},$$
$$\E\left(Z_{k}(n+1)-Z_k(n)|\F_n\right) =\frac{f_k(Z_1(n),...,Z_k(n))-Z_k(n)}{n+1}$$ $$+\frac{f_k(\tilde{Z}_1(n),...,\tilde{Z}_k(n))-f_k(Z_1(n),...,Z_k(n))}{n+1}.$$
Let
$$g_{k}(x_1,...,x_k)=f_k(x_1,...,x_k)-x_k$$
and
$$G_k(x_1,...,x_k)=\left(g_{1}(x_1),...,g_{k}(x_1,...,x_k)\right)^{t}.$$
We further use these functions to build multidimensional stochastic approximation process.
Note that $\bigtriangledown G_k$ is triangular matrix, and hence its eigenvalues are
$$\lambda_i=\frac{\partial f_i}{\partial x_i}-1=-\sum_{m=1}^{\infty}m\Pr(d=m)\frac{i+\beta}{2+\beta} \left(\sum_{j=1}^{i}x_j\frac{j+\beta}{2+\beta}\right)^{m-1}-1.$$
Therefore, for any $k$ and $(x_1,...,x_k)\in\Delta_k$
$$-L_k=\max_{1\leq i\leq k}\lambda_i<-1.$$
Consider the systems $G_l(x_1,...,x_l)=0$, $l=1,...,k$. Let us prove that these systems have a unique solution in $\Delta_l$. First, consider the case $l=1$. Note that $f_1(x_1)$ is strictly decreasing in $\Delta_1$, $f(0)=1$. Therefore, there is a unique solution $x^{*}_1$ of the equation $g_1(x_1)=0$ in $\Delta_1$. Second, prove that if the system $G_{l-1}(x_1,...,x_{l-1})=0$ has a unique solution $\rho^{*}_{l-1}=(x^{*}_1,...,x^{*}_{l-1})$ in $\Delta_{l-1}$, then the system $G_{l}(x_1,...,x_{l})=0$ has a unique solution $\rho^{*}_{l}=(x^{*}_1,...,x^{*}_{l})$ in $\Delta_{l}$. 
Since $f_{l}(\rho^{*}_{l-1},x_l)$ is decreasing function of $x_l$, $f_{l}(\rho^{*}_{l-1},0)>0$ and $f_l(\rho^*_{l-1},x_l)\to-\infty$ as $x_l\to\infty$, there is a unique solution $x^*_l\in\Delta_l$ of the equation $g_{l}(\rho^*_{l-1},x_l)=0$.
Therefore, the system $G_k(x_1,...,x_k)=0$ has a unique solution $\rho_k^{\ast}=(x_1^{\ast},...,x_{k}^{\ast})$ in $\Delta_k$.
From now, let $\Delta_k=[0,M_k]^k$, where $M_k>0$ is such that $\{\rho^*_k\}\cup[0,2]^k\subset\Delta_k$.

Set 
$$\gamma_n=\frac{1}{n+1},\;\tilde{\gamma}_n=\frac{1}{n+1+\frac{\beta}{2+\beta}},$$
$$W_k(n)=(Z_1(n),...,Z_k(n))^{t},\;\tilde{W}_k(n)=(\tilde{Z}_1(n),...,\tilde{Z}_k(n))^{t}.$$
Then
$$W_{k}(n+1)-W_k(n)=\gamma_n\left(G_k(W_k(n))+E_k(n+1)+R_k(n+1)\right),$$
where 
$$R_k(n+1)=\left[f_1(\tilde{Z}_1(n))-f_1(Z_1(n)), ..., f_k(\tilde{Z}_1(n),...,\tilde{Z}_k(n))-f_k(Z_1(n),...,Z_k(n))\right]^t,$$ $$E_k(n+1)=(W_k(n+1)-\E(W_k(n+1)|\F_n))/\gamma_n$$
and
$$\tilde{W}_k(n+1)-\tilde{W}_k(n)=\tilde{\gamma}_n\left(G_k(\tilde{W}_k(n))+ (\tilde{W}_k(n+1) -\E(\tilde{W}_k(n+1)|\F_n))/\tilde{\gamma}_n\right).$$
Below, we prove that $\tilde{W}_k(n)$ satisfies the conditions of Theorem~1.4.26 from \cite{Duflo}. 
Note that the process $\tilde{W}_k(n)$ belongs to $\Delta_k$, so conditions would be written for $x\in\Delta_k$. The first condition becomes $(G_k(x))^t(x-\rho_k^{\ast})<0$ for any $x\in\Delta_k,$ $x\neq\rho_k^{\ast}$. This conditions is clearly satisfied since $\bigtriangledown G_k(x)$ has negative eigenvalues for all $x\in\Delta_k$. The conditions on $\tilde{\gamma}_n$ ($\sum \tilde{\gamma}_n=\infty$, $\sum \tilde{\gamma}^2_n<\infty$) hold for our choice of $\tilde{\gamma}_n$.   The last condition, obviously, holds, if
$$||G_k(\tilde{W}_k(n))||^2+ ||(\tilde{W}_k(n+1) -\E(\tilde{W}_k(n+1)|\F_n))/\tilde{\gamma}_n||^2<K\quad \text{a.s.}$$
for some constant $K>0$. Since $G_k(x)$ is bounded in $\Delta_k$, it is enough to show that an absolute value of each coordinate of $(\tilde{W}_k(n+1) -\E(\tilde{W}_k(n+1)|\F_n))/\tilde{\gamma}_n$ is bounded by a constant. For $1\leq i\leq k$,
$$\left(n+1+\frac{\beta}{2+\beta}\right)\left|\tilde{Z}_i(n+1)-\E(\tilde{Z}_i(n+1)|\F_n)\right|$$
$$\leq \left(n+1+\frac{\beta}{2+\beta}\right) \left(\left|\tilde{Z}_i(n+1)-\tilde{Z}_i(n)\right|+\left|\E(\tilde{Z}_i(n+1)-\tilde{Z}_i(n)|\F_n)\right|\right)$$
$$=\left|(N_k(n+1)-N_k(n))-\frac{N_k(n)} {\left(n+\frac{\beta}{2+\beta}\right)}\right|+\left|f_k(\tilde{Z}_1(n),...,\tilde{Z}_k(n))-\tilde{Z}_k(n)\right|$$
$$\leq\left|(N_k(n+1)-N_k(n))\right|+\left|\frac{N_k(n)} {\left(n+\frac{\beta}{2+\beta}\right)}\right|+ \left|f_k(\tilde{Z}_1(n),...,\tilde{Z}_k(n))\right|+\left|\tilde{Z}_k(n)\right|\leq 4+\epsilon$$
for any $\epsilon>0$ and $n$ large enough.
So we get that $\tilde{W}_{k}(n)\to \rho_k^{\ast}$ almost surely by Theorem~1.4.26 of \cite{Duflo}, and, hence, $W_{k}(n)\to \rho_k^{\ast}$ almost surely as well.

Now we check that assumptions of Theorem 2.1 of \cite{Fort} holds, which gives the convergence in distribution of $\gamma^{-1/2}(W_k(n)-\rho_k^{\ast})$ under the conditional probability $\Pr(\cdot|W_{k}(n)\to \rho_k^{\ast})$. Note that the condition could be removed since $\Pr(W_{k}(n)\to \rho_k^{\ast})=1$. Let us describe assumptions of this theorem for our model.

Assumptions C1(a,b,c) claim that $\rho^{\ast}$ belongs to interior of $\Delta_k$, $G_k(x)$ is twice continuously differentiable in a neighborhood of $\rho^{\ast}$ and $\bigtriangledown G_k(x)$ has negative eigenvalues. This assumptions holds as was discussed above with the largest eigenvalue $-L_k<-1$.

Assumptions C2(a,b,c) could be rewritten in the following way (we make the second part even stronger than needed): $\E(E_k(n+1)|\F_n)=0$, $||E_k(n+1)||<C$ for some constant $C>0$ and $\E(E_k(n+1)E_k^t(n+1)|\F_n)=U_k^{\ast}+o(1)$
for some symmetric positive definite (random) matrix $U_k^{\ast}$. We will check these assumptions later.

Assumption C3 claims that $\gamma_n^{-1/2} R_k(n+1)=O_{w.p.1}(1)o_{L^1}(1)$, which holds from the definition of  $R_k(n+1)$ since its components $f_i(x_1,...,x_i)$, $i=1,...,k$, have bounded derivatives in $\Delta_k$.

Assumption C4(b) states that $\sum \gamma_n=\infty$, $\sum \gamma^2_n<\infty$ and $\log(\gamma_{n-1}/\gamma_n)\sim \gamma_n/\gamma^{\ast}$ for some $\gamma^{\ast}>1/(2L_k)$. Since $L_k>1$, it holds with $\gamma_{\ast}=1$.

To check assumptions C2, note that for any $n$
$$||E_k(n+1)||\leq\sum_{i=1}^{k}(n+1)|Z_k(n+1)-\E(Z_k(n+1)|\F_n)|$$
$$\leq(n+1)\sum_{i=1}^{k}\left(|Z_k(n+1)-Z_k(n)|+|\E(Z_k(n+1)-Z_k(n)|\F_n)|\right)$$
$$=\sum_{i=1}^{k}\left(|N_{k}(n+1)-N_k(n)|+|Z_k(n)| +|f_k(\tilde{Z}_1(n),...,\tilde{Z}_k(n))-Z_k(n)|\right)\leq 4k.$$
Now it remains to check that 
$$\E(E_k(n+1)E_k^t(n+1)|\F_n)=U_k^{\ast}+o(1)$$
for some symmetric positive definite (random) matrix $U_k^{\ast}$. Since components of $E_k(n+1)$ is not linearly dependent, $\E(E_k(n+1)E_k^t(n+1)|\F_n)$ is positive definite. So we have to prove that it can be decomposed as $U_k^{\ast}+o(1)$ where $U_k^{\ast}$ does not depend on $n$.
Note that from definition of $E_k(n+1)$ we have 
$$E_k(n+1)=(N_1(n+1)-\E(N_1(n+1)|\F_n),...,N_k(n+1)-\E(N_k(n+1)|\F_n))^t$$
$$=[(N_i(n+1)-N_i(n))-\E(N_i(n+1)-N_i(n)|\F_n)]_{1\leq i\leq k}^t.$$
Therefore,
$$\E(E_k(n+1)E_k^t(n+1)|\F_n)$$
$$=\E\left([(N_i(n+1)-N_i(n)-\E(N_i(n+1)-N_i(n)|\F_n))\right.$$
$$\times\left.(N_j(n+1)-N_j(n)-\E(N_j(n+1)-N_j(n)|\F_n))]_{1\leq i,j\leq k}|\F_n\right)$$
$$=[\E((N_i(n+1)-N_i(n))(N_j(n+1)-N_j(n))|\F_n)$$
$$-\E(N_i(n+1)-N_i(n)|\F_n)\E(N_j(n+1)-N_j(n))|\F_n]_{1\leq i,j\leq k}.$$ 
Note that
$$\E(N_i(n+1)-N_i(n)|\F_n)=f_i(\tilde{Z}_1(n),...,\tilde{Z}_i(n)) =f_i(x_1^{\ast},...,x_i^{\ast})+o(1)\quad \text{a.s.}$$
To calculate the expectations $\E((N_i(n+1)-N_i(n))(N_j(n+1)-N_j(n))|\F_n)$ we use formula \eqref{eq:N_1} and \eqref{eq:N_k}. Let $i\leq j$. As in calculation of $\Pr(N_i(n+1)-N_i(n)=1|\F_n)$, the case $i=1$ different from the case $i>1$. For $i=1$ we get that 
$$\E((N_1(n+1)-N_1(n))(N_j(n+1)-N_j(n))|\F_n)=\Pr(\deg Y_n(n)=j-1|\F_n)$$
$$-\Pr(\deg Y_n(n)=j|\F_n) \text{ if }j>2,$$
$$\E((N_1(n+1)-N_1(n))(N_j(n+1)-N_j(n))|\F_n)=0 \text{ if }j=2,$$
$$\E((N_1(n+1)-N_1(n))(N_j(n+1)-N_j(n))|\F_n)=\Pr(\deg Y_n(n)\neq 1|\F_n)\text{ if }j=1.$$
For $i>1$ we get that
$$\E((N_i(n+1)-N_i(n))(N_j(n+1)-N_j(n))|\F_n)=-\Pr(\deg Y_n(n)=i|\F_n) \text{ if }i=j-1,$$
$$\E((N_i(n+1)-N_i(n))(N_j(n+1)-N_j(n))|\F_n)=0 \text{ if }i<j-1,$$
$$\E((N_i(n+1)-N_i(n))(N_j(n+1)-N_j(n))|\F_n)=\Pr(\deg Y_n(n)=i-1|\F_n)$$
$$+\Pr(\deg Y_n(n)=i|\F_n)\text{ if }i=j.$$

Introduce functions
$$ 
f^k_{i,j}(x_1,...,x_k)=$$
$$
=\begin{cases}
f_1(x_1)&\;i=j=1;\\
f_j(x_1,...,x_j)& \;i=1,j>2;\\
\sum_{m=1}^{\infty}\Pr(d=m)\left(\left(\sum_{l=1}^{i}x_{l}\frac{(l+\beta)}{2+\beta}\right)^m- \left(\sum_{l=1}^{i-1}x_{l}\frac{(l+\beta)}{2+\beta}\right)^m\right)&\;i=j-1>1;\\
\sum_{m=1}^{\infty}\Pr(d=m)\left(\left(\sum_{l=1}^{i}x_{l}\frac{(l+\beta)}{2+\beta}\right)^m- \left(\sum_{l=1}^{i-2}x_{l}\frac{(l+\beta)}{2+\beta}\right)^m\right)&\;i=j>2;\\
\sum_{m=1}^{\infty}\Pr(d=m)\left(\frac{x_1(1+\beta)+x_2(2+\beta)}{2+\beta}\right)^m&\;i=j=2;\\
0&\text{ otherwise.}
\end{cases}
$$
For $i>j$ put $f^k_{i,j}=f^k_{j,i}$. 
From \eqref{eq:Y_h_k} and definition of $\tilde{Z}_k(n)$ we get 
$$\E((N_i(n+1)-N_i(n))(N_j(n+1)-N_j(n))|\F_n) =f^k_{i,j}(\tilde{Z}_1(n),...,\tilde{Z}_k(n)).$$ 
Note that since $\E C^d<\infty$ for any $C>0$, functions $f^k_{i,j}$ have bounded derivatives in $\Delta_k$. Hence,
$$\E((N_i(n+1)-N_i(n))(N_j(n+1)-N_j(n))|\F_n) f^k_{i,j}(x_1^{\ast},...,x_k^{\ast})+o(1)\quad \text{a.s.}$$
Therefore, if we define $U_k^{\ast}$ as
$$U_k^{\ast}=[f^k_{i,j}(x_1^{\ast},...,x_k^{\ast}) -f_i(x_1^{\ast},...,x_i^{\ast})f_j(x_1^{\ast},...,x_j^{\ast})]_{1\leq i,j\leq k},$$
we would get that
$$\E(E_k(n+1)E_k^t(n+1)|\F_n)=U_k^{\ast}+o(1).$$
Hence, we could apply Theorem 2.1 of \cite{Fort}. So we proved the following theorem.
\begin{theorem}
\label{thm:fixed_degree}
Let $\E C^d<\infty$ for any $C>0$. Then, under above notations, $\gamma_n^{-1/2}(W_k(n)-\rho_k^{\ast})$ converges in distribution to a random vector with a characteristic function given for any $x\in\R^k$ by
$$\exp\{-\frac{1}{2}x^tVx\},$$
where matrix $V$ is the positive definite matrix satisfying
$$V(I+2\gamma^{\ast}\triangledown G_k(\rho_k^{\ast})^t)+(I+2\gamma^{\ast}\triangledown G_k(\rho_k^{\ast}))V=-2\gamma^{\ast}U_k^{\ast}.$$
\end{theorem}

\section{Discussion}

In Section 4, we use a stochastic approximation technique in order to simplify the proof of Lemma~\ref{lem:zero_set}. This lemma could be proved by the martingale technique only (which is later used in the cases $\E d<2+\beta$ and $\E d=2+\beta$). In the one-dimensional case, the stochastic approximation result we use could be obtained from convergence theorems for martingales. In two last cases, usage of the stochastic approximation does not provide many benefits, although it probably can be applied as well. 

Note that conditions on $d$ in Theorem~\ref{thm:max_degree} are needed only in the case $\E d\leq 2+\beta$ and could be lightened after a more detailed analysis of the expression on the right-hand side of (\ref{p_n}).

In Section 5, we use a multidimensional stochastic approximation (due to a recursive dependence of $N_{i}(n)$). A similar technique could be used to get a central limit theorem for the number of vertices of the fixed degree for the min-choice model (see \cite{MP13}) and, probably, to its extension to meek choice in \cite{HJ16}.\\

Note that the sequence $(x_1^{\ast},x_2^{\ast},...)$, which is defined in Section 5, defines a limit of the empirical degree distribution. Variations of $\beta$ and the distribution of $d$ would result in different limits. It could be interesting to consider reverse problem, which is to try to reconstruct $\beta$ and $d$ from a given sequence $(x^*_1,x^*_2,...)$, that satisfy some conditions (with basic conditions $0<x^*_{i+1}<x^*_i$, $\sum_{i=1}^{\infty}x^*_i= 1$ and $\sum_{i=1}^{\infty}x^*_i(i+\beta)\leq 2+\beta$ for some $\beta>-1$). 

It could be easily shown that $\sum_{i=1}^{\infty}x^*_i= 1$, but $\sum_{i=1}^{\infty}x^*_i(i+\beta)$ may be less than $2+\beta$. In fact, if $\E d>2+\beta$, $\sum_{i=1}^{k}x^*_i(i+\beta)\leq 2+\beta-x^*$ for all $k$. Furthermore, one could consider the sum $S_{k}=\sum_{i=1}^{k}x^*_i\frac{i+\beta}{2+\beta}$ (similar to \cite{HJ16}). Since it is increasing and bounded from above by $1$, there is a limit $S_{\ast}=\lim_{k\to\infty}S_k$. Similar to \cite{HJ16}, one could ask is $S_{\ast}=1$ if $\E d\leq 2+\beta$ and $S_{\ast}=1-x^{\ast}/(2+\beta)$ if $\E d> 2+\beta$, or there is some additional weight loss.

Moreover, the reverse problem could be reduced to finding a nonnegative solution of infinite system
$$\sum_{m=1}^{\infty}S_{k}^{m}p_m=k-\sum_{i=1}^k(k+1-i)x_{i}^{\ast}\quad k\in\N$$ 
of equations on $p_m$, where $p_m=\mathbb{P}(d=m)$. Note that this system has Vandermond matrix, so one could suggest that there is a solution (if $S_i\neq S_j$, $i\neq j$, which equals to condition $x_i^{\ast}>0$ for all $i\in\N$), but it could be not a nonnegative one. It also could be interesting to see if some types of asymptotic behaviour of $x_{i}^{\ast}$ (e.g., $x_{i}^{\ast}\sim\frac{c}{i^{\gamma}}$ for $1<\gamma<2$) would result in absence of a nonnegative solution.

Theorem~\ref{thm:fixed_degree} provides the central limit theorem for any fixed $k$. It could be interesting to consider $k$ depending on $n$ and study an asymptotic behaviour of $N_k(n)$. This study can be motivated by the fact that  $x_k^{\ast}$ decays quite fast with $k$, and so does the smallest eigenvalue of $V$, which reduces an efficiency of the central limit theorem. \\

Theorem~\ref{thm:max_degree} gives us an asymptotic of the maximal degree of the graph. For the subcritical case ($\mathbb{E}d<2+\beta$), it behaves similarly to the maximal degree in the Mori's preferential attachment model. Therefore, it is interesting to compare the respective degree distributions. In contrast to the standard preferential attachment model, in max-choice model, the probability to increase a degree of a fixed vertex depends not only on the degree, but also on its position in the joint degree distribution. Indeed, let
$$A(m,n)=\sum_{j=1}^{n+1}(\deg v_j(n)+\beta)\1\{\deg v_j(n)>m\},$$
$$A_i(n)=A(\deg v_i(n),n),$$
$$B(m,n)=\sum_{j=1}^{n+1}(\deg v_j(n)+\beta)\1\{\deg v_j(n)=m\},$$
$$B_i(n)=B(\deg v_i(n),n),$$
$$\alpha(x,n)=\sum_{k=1}^{\infty}\mathbb{P}(d=k)\left(1-\frac{x}{(2+\beta)n+\beta}\right)^k,$$
$$a(m,n)=\alpha(A(m,n),n)-\alpha(A(m,n)+B(m,n),n),$$
where $a(m,n)$ is the probability to increase a degree of a vertex with the degree $m>1$ at the step $n$. Therefore, for the vertex $v_i$, $i\leq n$, the probability $p_i(n)$ to increase its degree at the step $n$ is
$$p_i(n)= \sum_{l=1}^{\infty}
\1\{B_i(n)=l(\deg v_i(n)+\beta)\}\frac{a(\deg v_i(n),n)}{l}$$
$$=\sum_{l=1}^{\infty}
\1\{B_i(n)=l(\deg v_i(n)+\beta)\}\frac{a(\deg v_i(n),n)(\deg v_{i}(n)+\beta)}{B_i(n)}$$
$$=(\deg v_i(n)+\beta)\sum_{k=1}^{\infty}\frac{\mathbb{P}(d=k)}{B_i(n)}\left(\left(1-\frac{A_i(n)}{(2+\beta)n+\beta}\right)^k -\left(1-\frac{A_i(n)+B_i(n)}{(2+\beta)n+\beta}\right)^k\right)$$
$$=\frac{\deg v_i(n)+\beta}{(2+\beta)n+\beta}\sum_{k=1}^{\infty} \mathbb{P}(d=k)\sum_{j=0}^{k-1}\left(1-\frac{A_i(n)}{(2+\beta)n+\beta}\right)^j \left(1-\frac{A_i(n)+B_i(n)}{(2+\beta)n+\beta}\right)^{k-j-1}.$$
The multiplier in front of the sum in the last line equals to the probability to increase the degree of the vertex in Mori's preferential attachment model without choice, while the sum depends on a position in the joint degree distribution. It would be interesting to get some analogue of the power-law for the degree distribution (may be with a non-constant but bounded exponent).
\\

We consider only trees, but similar results should hold for standard generalizations of the model with more edges (e.g., $m$ edges are drawn at each step, $m$ could be random). Such generalizations are of interest because their graph properties are more natural. For example, commonly, the diameter of the model is close to the diameter of a real network for $m\geq 2$ (see, e.g., \cite{barabasi}, \cite{BR04}). To achieve such an effect in our model, one could consider models in which the choice is made with respect to some other criteria. For example, coin toss among vertices with the same degree could be replaced with a choice that minimizes a distance between the new vertex and the vertex with the highest degree.

\section*{Acknowledgements.}
The author is grateful to Professor Itai Benjamini for proposing the model and to Professor Maksim Zhukovskii for helpful discussions.
\bibliographystyle{alpha}
\bibliography{random number of choices_revised}

\end{document}